\documentclass[12pt,draftcls,onecolumn]{IEEEtran}

\usepackage{graphicx}

\ifCLASSINFOpdf
\else
\fi
\usepackage{amsmath}
\usepackage{amsthm}
\usepackage{amsfonts}
\usepackage{amssymb}
\usepackage{color}

\newtheorem{theorem}{Theorem}
\newtheorem{lemma}[theorem]{Lemma}
\newtheorem{defin}[theorem]{Definition}

\newtheorem{remark}[theorem]{Remark}

\begin{document}
%
\title{A simple approach to distributed observer design for linear systems}
%
%
%


\author{Weixin~Han,
        Harry~L.~Trentelman,~\IEEEmembership{Fellow,~IEEE,}
        Zhenhua~Wang,
        and~Yi~Shen,~\IEEEmembership{Member,~IEEE}
\thanks{This work was partially supported by China Scholarship Council and National Natural Science Foundation of China (Grant No. 61273162, 61403104).}
\thanks{Weixin Han, Zhenhua Wang and Yi Shen are with the Department of Control Science and Engineering, Harbin Institute of Technology, Harbin,
150001 P. R. China.
        {\tt\small zhenhua.wang@hit.edu.cn}}%
\thanks{Harry L. Trentelman is with the Johann Bernoulli Institute for Mathematics and Computer Science, University of Groningen,
   9700 AK Groningen The Netherlands.
        {\tt\small h.l.trentelman@rug.nl}}%
}

%
%

\markboth{Journal of \LaTeX\ Class Files,~Vol.~14, No.~8, August~2015}%
{Shell \MakeLowercase{\textit{et al.}}: Bare Demo of IEEEtran.cls for IEEE Journals}
%



\maketitle


\begin{abstract}
This note investigates the distributed estimation problem for continuous-time linear time-invariant (LTI) systems observed by a network of observers. Each observer in the network has access to only part of the output of the observed system, and communicates with its neighbors according to a given network graph. In this note we recover the known result that if the observed system is observable and the network graph is a strongly connected digraph, then a distributed observer exists. Moreover, the estimation error can be made to converge to zero at any a priori given decay rate. Our approach leads to a relatively straightforward proof of this result, using the mirror of the balanced graph associated with the original network graph. The numerical design of our distributed observer is reduced to solving linear matrix inequalities (LMI's). Each observer in the network has state dimension equal to that of the observed plant.
\end{abstract}


%
\IEEEpeerreviewmaketitle

\section{Introduction}
%
%
%
%
\IEEEPARstart{I}{n} recent years, distributed state estimation has received a lot of attention in response to increasing demand for estimating the state of a dynamic system over spatially deployed multiagents or networked sensors \cite{Kim2016CDC}. 
A linear dynamical system is monitored by a network of sensor nodes. The objective of each node is to asymptotically estimate the state of the dynamical system using its own (limited) measurements and via information exchange with neighbors. This is known as the distributed state estimation problem \cite{Mitra2016CCC}.
The main challenge in the distributed estimation problem comes from the limitation that no single observer can estimate the state of the system only using its local measurement. Thus, classical observer design methods cannot be directly applied to the distributed estimation problem \cite{Park2012ACC}.

A similar distributed estimation problem for linear dynamical systems has been studied with various approaches. These approaches can be classified into two main branches, namely: Kalman-filter based techniques, and observer based techniques. The distributed Kalman-filter based state estimation approach was first proposed in \cite{Olfati2005ECC,Olfati2007CDC}. These methods rely on a two-step strategy: a Kalman filter based state estimate update rule, and a data fusion step based on average-consensus \cite{Khan2011ACC}. In \cite{Khan2010CDC,Khan2014AUT}, the authors propose a single-time-scale strategy and design a scalar-gain estimator. Sufficient conditions for stability of the estimator are given in their work. However, the tight coupling between the network topology and the plant dynamics limits the range of the scalar gain parameter. 

On the other hand, an observer-based approach is studied under the joint observability assumption. In \cite{Park2012ACC,Park2012CDC,Park2017TAC}, a state augmentation observer is constructed to cast the distributed estimation problem as the problem of designing a decentralized stabilizing controller for an LTI plant, using the notion of fixed modes \cite{Anderson1981AUT}. It should be noted that these works only discuss discrete-time systems. In \cite{Wang2017TAC}, a general form of distributed observer with arbitrarily fast convergence rate was proposed. Based on the Kalman observable canonical decomposition, local Luenberger observers at each node are constructed in \cite{Mitra2016CCC,Mitra2016CDC,Mitra2017}. The observer reconstructs a certain portion of the state solely by using its own measurements, and uses consensus dynamics to estimate the unobservable portions of the state at each node. Specifically, in \cite{Kim2016CDC} two observer gains are designed to achieve distributed state estimation, one for local measurements and the other for the information exchange. 


In this note, we study the distributed estimation problem for continuous-time LTI systems observed by a network of Luenberger observers (see Fig. \ref{dobs} for an illustration). Each observer accesses a portion of the output of the observed known LTI system, and communicates with its neighboring observers. The observer at each node is designed to asymptotically estimate the state of the system. Unlike the state augmentation observer approach in \cite{Park2012ACC,Park2012CDC,Park2017TAC} and also the observer construction in \cite{Wang2017TAC}, in this note the local Luenberger observer at each node has the same order as the observed system. We borrow the idea of observability decomposition from \cite{Kim2016CDC}. The estimation error of the $i$-th observer in the observable part is stabilized by output injection, and the estimation in the unobservable part reaches a consensus with the other local observer's estimates. Compared with \cite{Kim2016CDC}, in this note the information among the local observers is exchanged by a strongly connected directed graph. We relax the design constraints and develop a new simple design procedure. The choice of the gain matrices in the distributed observer becomes more flexible. We decouple the topological information of the network from the local gain matrices of the observer by the introduction of an auxiliary undirected graph. This auxiliary graph is obtained by balancing the original graph and then taking the mirror of this balanced graph, see \cite{Olfati2004TAC}. The gain matrices in the observer can be obtained by solving linear matrix inequalities (LMI's), which makes the distributed observer design numerically feasible.

\begin{figure}
  \centering
  \includegraphics[scale=0.45]{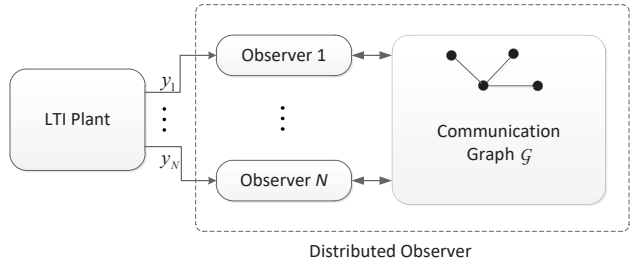}
  \caption{Framework for distributed state estimation}
  \label{dobs}
\end{figure}

%
%

\section{Preliminaries and Problem Formulation}
\subsection{Preliminaries}
\noindent \emph{Notation}:
For a given matrix $M$, its transpose is denoted by $M^T$ and $M^{-1}$ denotes its inverse. The symmetric part of a square real matrix $M$ is sometimes denoted by $\mathrm{Sym}(M):=M+M^T$. The rank of the matrix $M$ is denoted by $\mathrm{rank}~M$. The identity matrix of dimension $N$ will be denoted by $I_N$. The vector $\mathbf{1}_N$ denotes the $N\times 1$ column vector comprising of all ones. For a symmetric matrix $P$, $P>0$ $(P<0)$ means that $P$ is positive (negative) definite. For a set $\{A_1,A_2,\cdots,A_N\}$ of matrices, we use $\mathrm{diag}\{A_1,A_2,\cdots,A_N\}$ to denote the block diagonal matrix with the $A_i$'s along the diagonal, and the matrix $\begin{bmatrix}
A_1^T &A_2^T&\cdots &A_N^T
\end{bmatrix}^T $ is denoted by $\mathrm{col}(A_1,A_2,\cdots,A_N)$. The Kronecker product of the matrices $M_1$ and $M_2$ is denoted by $M_1\otimes M_2$. For a linear map $A:\mathcal{X}\to \mathcal{Y}$, $\mathrm{ker}~A:=\{x\in \mathcal{X}| Ax=0\}$ and $\mathrm{im}~A:=\{Ax| x\in \mathcal{X}\}$ will denote the kernel and image of $A$, respectively. For a real inner product space $\mathcal{X}$, if $\mathcal{V}$ is a subspace of $\mathcal{X}$, then $\mathcal{V}^{\perp}$ will denote the orthogonal complement of $\mathcal{V}$.

In this note, a weighted directed graph is denoted by $\mathcal{G=(N,E,A)}$, where $\mathcal{N} = \{1,2,\cdots,N\}$ is a finite nonempty set of nodes, $\mathcal{E\subset N\times N}$ is an edge set of ordered pairs of nodes, and $\mathcal{A}=[a_{ij}]\in \mathbb{R}^{N\times N}$ denotes the adjacency matrix. The $(j,i)$-th entry $a_{ji}$ is the weight associated with the edge $(i,j)$. We have $a_{ji}\neq 0$ if and only if $(i,j)\in \mathcal{E}$. Otherwise $a_{ji}=0$. An edge $(i,j)\in \mathcal{E}$ designates that the information flows from node $i$ to node $j$. A graph is said to be undirected if it has the property that $(i,j)\in \mathcal{E}$ implies $(j,i)\in \mathcal{E}$ for all $i,j\in \mathcal{N}$. We will assume that the graph is simple, i.e., $a_{ii}=0$ for all $i\in \mathcal{N}$. For an edge $(i,j)$, node $i$ is called the parent node, node $j$ the child node and $j$ is a neighbor of $i$. A directed path from node $i_1$ to $i_l$ is a sequence of edges $(i_k, i_{k+1})$, $k=1,2,\cdots,l-1$ in the graph. A directed graph $\mathcal{G}$ is strongly connected if between any pair of distinct nodes $i$ and $j$ in $\mathcal{G}$, there exists a directed path from $i$ to $j$, $i,j\in \mathcal{N}$.

The Laplacian $\mathcal{L}=[l_{ij}]\in \mathbb{R}^{N\times N}$ of $\mathcal{G}$ is defined as $\mathcal{L:=D-A}$, where the $i$-th diagonal entry of the diagonal matrix $\mathcal{D}$ is given by $d_i=\sum_{j=1}^N a_{ij}$. By construction, $\mathcal{L}$ has a zero eigenvalue with a corresponding eigenvector $\mathbf{1}_N$ (i.e., $\mathcal{L}\mathbf{1}_N=0_N$), and if the graph is strongly connected, all the other eigenvalues lie in the open right-half complex plane. 

For strongly connected graphs $\mathcal{G}$, we review the following lemma.
\begin{lemma}\cite{Olfati2004TAC,Ren2005TAC,Yu2010E3T}
Assume $\mathcal{G}$ is a strongly connected directed graph. Then there exists a unique positive row vector $r=\begin{bmatrix}
r_1,\cdots,r_N
\end{bmatrix} $ such that $r\mathcal{L}=0$ and $r\mathbf{1}_N=N$. Define $R:=\mathrm{diag}\{r_1,\cdots,r_N\}$. Then $\hat{\mathcal{L}}:=R\mathcal{L+L}^TR$ is positive semi-definite, $\mathbf{1}_N^T\hat{\mathcal{L}}=0$ and $\hat{\mathcal{L}}\mathbf{1}_N=0$.
\end{lemma}

We note that $R\mathcal{L}$ is the Laplacian of the balanced digraph obtained by adjusting the weights in the original graph. The matrix $\hat{\mathcal{L}}$ is the Laplacian of the undirected graph obtained by taking the union of the edges and their reversed edges in this balanced digraph. This undirected graph is called the mirror of this balanced graph \cite{Olfati2004TAC}.

\subsection{Problem formulation}

In this note, we consider the following continuous-time LTI system
\begin{equation} \label{sys0}
\begin{array}{l}
\dot{x}=Ax\\
y=Hx
\end{array}
\end{equation}
where $x\in \mathbb{R}^n$ is the state and $y\in \mathbb{R}^p$ is the measurement output. We partition the output $y$ as $y=\mathrm{col}(y_1,\cdots,y_N)$, where $y_i\in \mathbb{R}^{p_i}$ and $\sum_{i=1}^Np_i=p$. Accordingly, $H=\mathrm{col}(H_1,\cdots,H_N)$ with $H_i\in \mathbb{R}^{p_i\times n}$. Here, the portion $y_i=H_ix\in \mathbb{R}^{p_i}$ is assumed to be the only information that can be acquired by the node $i$.

In this note, a standing assumption will be that the communication graph is a strongly connected directed graph. We will also assume that the pair $(H,A)$ is observable. However, $(H_i,A)$ is not assumed to be observable or detectable.


We will design a distributed Luenberger observer for the system given by (\ref{sys0}) with the given communication network. The distributed observer will consist of $N$ local observers, and the local observer at node $i$ has the following dynamics
\begin{equation} \label{obsl}
\dot{\hat{x}}_i=A\hat{x}_i+L_i(y_i-H_i\hat{x}_i)+\gamma r_iM_i\sum_{j = 1}^N a_{ij}(\hat{x}_j-\hat{x}_i),~i\in \mathcal{N} 
\end{equation}
where $\hat{x}_i\in \mathbb{R}^n$ is the state of the local observer at node $i$, $a _{ij}$ is the $(i,j)$-th entry of the adjacency matrix $\mathcal{A}$ of the given network, $r_i$ is defined as in Lemma 1, $\gamma \in \mathbb{R}$ is a coupling gain to be designed, and $L_i\in \mathbb{R}^{n\times p_i}$ and $M_i\in \mathbb{R}^{n\times n}$ are gain matrices to be designed.

The objective of distributed state estimation is to design a network of observers that cooperatively estimate the state of the system described by system (\ref{sys0}). We shall borrow the following definition from \cite{Park2017TAC} for our analysis. 
\begin{defin}
A distributed observer achieves \emph{omniscience asymptotically} if for all initial conditions on (\ref{sys0}) and (\ref{obsl}) we have
\begin{equation}\label{esg}
\lim_{t\to \infty}\|\hat{x}_i(t)-x(t)\|=0
\end{equation}
for all $i\in\mathcal{N}$, i.e., the state estimate maintained by each node asymptotically converges to the true state of the plant.
\end{defin}

To analyze and synthesize observer (\ref{obsl}), we define the local estimation error of the $i$-th observer as
\begin{equation}
e_i:=\hat{x}_i-x.
\end{equation}
 
Combining (\ref{sys0}) and (\ref{obsl}) yields the following error equation
\begin{equation}
\dot{e}_i=(A-L_iH_i)e_i+\gamma r_iM_i\sum_{j = 1}^N a_{ij}(e_j-e_i).
\end{equation}
Let $e:=\mathrm{col}(e_1^T,e_2^T,\cdots,e_N^T)$ be the joint vector of errors. Then we have the global error system
\begin{equation}\label{sys_eg}
\dot{e}=\Lambda e-\gamma \overline{M}(R\mathcal{L}\otimes I_n)e,
\end{equation}
where
\begin{equation*}
\Lambda=\mathrm{diag}\{A-L_1H_1,\cdots,A-L_NH_N\},
\end{equation*}
\begin{equation*}
\overline{M}=\mathrm{diag}\{M_1,\cdots,M_N\},
\end{equation*}
and $R$ is as defined in Lemma 1.  

The distributed observer achieves omniscience asymptotically (\ref{esg}) if and only if the global error system (\ref{sys_eg}) is stable.

Since $(H_i,A)$ is not necessarily observable or detectable, $L_i$ cannot be designed using any classical method directly. We use an orthogonal transformation that yields the observability decomposition for the pair $(H_i,A)$. For $i \in \mathcal{N}$, let $T_i$ be an orthogonal matrix, i.e., a square matrix such that $T_iT_i^T=I_n$, such that the matrices $A$ and $H_i$ are transformed by the state space transformation $T_i$ into the form 
\begin{equation} \label{trf}
T_i^TAT_i=\begin{bmatrix}
A_{io} &0\\
A_{ir} &A_{iu}
\end{bmatrix},~H_iT_i=\begin{bmatrix}
H_{io} &0
\end{bmatrix}
\end{equation}
where $A_{io}\in \mathbb{R}^{v_i\times v_i}$, $A_{ir}\in \mathbb{R}^{(n-v_i)\times v_i}$, $A_{iu}\in \mathbb{R}^{(n-v_i)\times (n-v_i)}$, $H_{io}\in \mathbb{R}^{p_i\times v_i}$, and $n-v_i$ is the dimension of the unobservable subspace of the pair $(H_i,A)$. Clearly, by construction, the pair $(H_{io},A_{io})$ is observable.
In addition, if we partition $T_i=\begin{bmatrix}
T_{i1} &T_{i2}
\end{bmatrix} $, where $T_{i1}$ consists of the first $v_i$ columns of $T_i$, then the unobservable subspace is given by $\mathrm{im}~T_{i2}=\mathrm{ker}~O_i$,
where $O_i=\mathrm{col}(H_i,H_iA,\cdots,H_iA^{n-1})$. Note that $\mathrm{im}~T_{i1}=\mathrm{ker}~O_i^{\perp}$.

\section{Main Results}

In this section, we study the distributed observer design. Before presenting the main design procedure, we state the following lemmas, based on Lemma 1. Our first lemma is standard:
\begin{lemma}\cite{Li2014book}
For a strongly connected directed graph $\mathcal{G}$, zero is a simple eigenvalue of $\hat{\mathcal{L}}=R\mathcal{L+L}^TR$ introduced in Lemma 1. Furthermore, its eigenvalues can be ordered as $\lambda _1=0<\lambda _2\leqslant \lambda _3\leqslant \cdots \leqslant \lambda _N $. Furthermore, there exists an orthogonal matrix $U=\begin{bmatrix} \frac{1}{\sqrt{N}}\mathbf{1}_N &U_2 \end{bmatrix} $, where $U_2\in \mathbb{R}^{N \times (N-1)}$, such that $U^T(R\mathcal{L+L}^TR)U=\mathrm{diag}\{0,\lambda _2,\cdots,\lambda _N\}$.
\end{lemma}

\begin{lemma}
Let $\mathcal{L}$ be the Laplacian matrix associated with the strongly connected directed graph $\mathcal{G}$. For all $g_i>0$, $i\in \mathcal{N}$, there exists $\epsilon>0$ such that
\begin{equation} \label{L1}
T^T((R\mathcal{L+L}^TR)\otimes I_n)T+G>\epsilon I_{nN},
\end{equation}
where $T=\mathrm{diag}\{T_1,\cdots,T_N\}$, $R$ is defined as in Lemma 1, $G=\mathrm{diag}\{G_1,\cdots,G_N\}$, and $G_i=\begin{bmatrix}
 g_i I_{v_i} & 0\\
 0 &0_{n-v_i}
\end{bmatrix}$, $i\in \mathcal{N}$.
\end{lemma}
\begin{proof}
The inequality (\ref{L1}) holds if and only if the following inequality holds.
\begin{equation} \label{L2}
(U^T(R\mathcal{L+L}^TR)U)\otimes I_n+(U^T\otimes I_n)TGT^T(U\otimes I_n)>0,
\end{equation}
where $U$ is as in Lemma 3.

The inequality (\ref{L2}) holds if the following inequality holds.
\begin{equation} \label{L3}
\lambda _2 I_{nN}-(U\otimes I_n)\begin{bmatrix}
\lambda _2 I_n &0\\
0 &0_{n(N-1)}
\end{bmatrix}(U^T\otimes I_n)+TGT^T>0.
\end{equation}
Since $U=\begin{bmatrix}
\frac{1}{\sqrt{N}}\mathbf{1}_N &U_2
\end{bmatrix} $ and $U^T=\begin{bmatrix}
\frac{1}{\sqrt{N}}\mathbf{1}_N^T \\
U_2^T
\end{bmatrix} $, the inequality (\ref{L3}) is equivalent to
\begin{equation} \label{L4}
\lambda _2 I_{nN}-\frac{\lambda _2}{N}(\mathbf{1}_N\otimes I_n)(\mathbf{1}_N^T\otimes I_n)+TGT^T>0.
\end{equation}

By pre- and post- multiplying with $T^T$ and $T$, the inequality (\ref{L4}) is equivalent to
\begin{equation}
\lambda _2 I_{nN}-\frac{\lambda _2}{N}T^T(\mathbf{1}_N\otimes I_n)(\mathbf{1}_N^T\otimes I_n)T+G>0,
\end{equation}
that is 
\begin{equation} \label{L5}
\lambda _2 I_{nN}+G-\frac{\lambda _2}{N}\begin{bmatrix}
T_1 &\cdots &T_N
\end{bmatrix}^T\begin{bmatrix}
T_1 &\cdots &T_N
\end{bmatrix}>0.
\end{equation}
The orthogonal matrix $T_i$ can be partitioned as $T_i=\begin{bmatrix}
T_{i1} &T_{i2}
\end{bmatrix}$, $i\in \mathcal{N}$. It is clear that $T_{i1}T_{i1}^T+T_{i2}T_{i2}^T=I_n$.
By using the Schur complement lemma \cite{Boyd1994}, the inequality (\ref{L5}) is equivalent to
\begin{equation} \label{L51}
\begin{bmatrix}
\Psi _1  &\cdots &0  &T_{1}^T\\
\vdots  &\ddots &\vdots&\vdots\\
0 &\cdots &\Psi _N &T_{N}^T\\
T_{1}  &\cdots  &T_{N}&\frac{N}{\lambda _2} I_n
\end{bmatrix}>0.
\end{equation}
where $\Psi _i=\begin{bmatrix}
(\lambda _2+g_i)I_{v_i} &0 \\
0 &\lambda _2 I_{n-v_i}
\end{bmatrix}$, $i\in \mathcal{N}$.
Again using the Schur complement lemma, this is equivalent with
\begin{equation}\label{L6}
\begin{bmatrix}
\lambda _2 I_{n-v_1} &0 &\cdots&0  &T_{12}^T\\
0 &\Psi _2  &\cdots &0  &T_{2}^T\\
\vdots  &\vdots  &\ddots &\vdots&\vdots\\
0 &0 &\cdots &\Psi _N  &T_{N}^T\\
T_{12} &T_{2} &\cdots  &T_{N} &\frac{N}{\lambda _2} I_n-\frac{1}{\lambda _2 +g_1}T_{11}T_{11}^T
\end{bmatrix}>0.
\end{equation}
By repeatedly using the Schur complement lemma, we finally obtain that inequality (\ref{L5}) holds if and only if 
\begin{equation}\label{L8}
\frac{N}{\lambda _2} I_n-\sum_{i=1}^N \frac{1}{\lambda _2 +g_i}T_{i1}T_{i1}^T-\sum_{i=1}^N\frac{1}{\lambda _2 }T_{i2}T_{i2}^T>0.
\end{equation}

The left-hand side of inequality (\ref{L8}) is equal to
\begin{equation}
\begin{array}{ll}
&\frac{N}{\lambda _2} I_n-\sum_{i=1}^N \frac{1}{\lambda _2 +g_i}T_{i1}T_{i1}^T-\sum_{i=1}^N\frac{1}{\lambda _2 }T_{i2}T_{i2}^T\\
=&\frac{N}{\lambda _2} I_n-\sum_{i=1}^N\frac{1}{\lambda _2 }T_{i2}T_{i2}^T-\sum_{i=1}^N\frac{1}{\lambda _2 }T_{i1}T_{i1}^T\\
&+\sum_{i=1}^N\frac{1}{\lambda _2 }T_{i1}T_{i1}^T-\sum_{i=1}^N \frac{1}{\lambda _2 +g_i}T_{i1}T_{i1}^T\\
=&\sum_{i=1}^N (\frac{1}{\lambda _2}-\frac{1}{\lambda _2 +g_i})T_{i1}T_{i1}^T\\
\geqslant &\sum_{i=1}^N (\frac{1}{\lambda _2}-\frac{1}{\lambda _2 +g_{min}})T_{i1}T_{i1}^T\\
=&(\frac{N}{\lambda _2}-\frac{N}{\lambda _2 +g_{min}})\begin{bmatrix}
T_{11} &\cdots & T_{N1}
\end{bmatrix} \begin{bmatrix}
T_{11} &\cdots & T_{N1}
\end{bmatrix}^T,
\end{array}
\end{equation}
where $g_{min}$ is the minimum value of $g_i$, $i\in \mathcal{N}$. Obviously, we have $(\frac{N}{\lambda _2}-\frac{N }{\lambda _2 +g_{min}})>0 $ since $g_{min}>0$.

We will now prove that $\mathrm{rank}\begin{bmatrix}
T_{11} & T_{21} &\cdots & T_{N1}
\end{bmatrix}=n$, so that it has full row rank.

Indeed, for $T_{i1}$, we have
\begin{equation} \label{L9}
\mathrm{im}~T_{i1}=\mathrm{ker}~O_i^{\perp},
\end{equation}
where $O_i=\mathrm{col}(H_i,H_iA,\cdots,H_iA^{n-1})$.

Hence,
\begin{equation}
\begin{array}{ll}
&\mathrm{im}\begin{bmatrix}
T_{11} & T_{21} &\cdots & T_{N1}
\end{bmatrix}^{\perp}\\
=&(\mathrm{im}~T_{11}+ \mathrm{im}~T_{21}+\cdots \mathrm{im}~T_{N1})^{\perp}\\
=&\bigcap_{i=1}^N\mathrm{im}~T_{i1}^{\perp}\\
=&\bigcap_{i=1}^N \mathrm{ker}~O_i\\
=&\mathrm{ker}\begin{bmatrix}
O_1\\
\vdots\\
O_N
\end{bmatrix}\\
=&0,
\end{array}
\end{equation}
where we have used our standing assumption that the pair $(H,A)$ is observable.
This implies
\begin{equation}
\mathrm{rank}\begin{bmatrix}
T_{11} & T_{21} &\cdots & T_{N1}
\end{bmatrix}=n.
\end{equation}

Consequently, $\begin{bmatrix}
T_{11} & T_{21} &\cdots & T_{n1}
\end{bmatrix}$ has full row rank $n$, so we obtain:
\begin{equation}
(\frac{N}{\lambda _2}-\frac{N}{\lambda _2 +g_{min}})\begin{bmatrix}
T_{11} &\cdots & T_{N1}
\end{bmatrix} \begin{bmatrix}
T_{11} &\cdots & T_{N1}
\end{bmatrix}^T>0.
\end{equation}

We conclude that the left-hand side of (\ref{L1}) is positive definite, and consequently, for any choice of $g_i>0$, $i\in \mathcal{N}$, there exists a scalar $\epsilon>0$ such that inequality (\ref{L1}) holds.

\end{proof}

The following lemma now deals with the existence of a distributed observer of the form (\ref{obsl}) that achieves omniscience with an a priori given error decay rate. A condition for its existence is expressed in terms of solvability of an LMI. Solutions to the LMI yield required gain matrices. Let $r_i>0$, $i\in \mathcal{N}$, be as in Lemma 1. Let $g_i>0$, $i\in \mathcal{N}$, and $\epsilon >0$ be such that (\ref{L1}) holds. Let $\gamma \in \mathbb{R}$. Finally, let $\alpha>0$ be the desired error decay rate. We have the following:
\begin{lemma}
There exist gain matrices $L_i$ and $M_i$, $i\in \mathcal{N}$, such that the distributed observer (\ref{obsl}) achieves omniscience asymptotically (\ref{esg}) and all solutions of the error system (\ref{sys_eg}) converge to zero with decay rate at least $\alpha$ if there exist positive definite matrices $P_{io}\in \mathbb{R}^{v_i\times v_i},P _{iu}\in \mathbb{R}^{(n-v_i)\times (n-vi)}$, and a matrix $W_i\in \mathbb{R}^{v_i\times p_i}$ such that
\begin{equation} \label{Th1}
\begin{bmatrix}
\Phi _i+ \gamma  g_i I_{v_i}  &A_{ir}^TP_{iu}\\
 P_{iu}A_{ir} &\mathrm{Sym}(P_{iu}A_{iu})+2\alpha P_{iu}
\end{bmatrix}-\gamma \epsilon I_n<0,~\forall i\in \mathcal{N},
\end{equation}
where $\Phi _i:= P_{io}A_{io}+A_{io}^TP_{io}-W_iH_{io}-H_{io}^TW_i^T+2\alpha P_{io}$. In that case, the gain matrices in the distributed observer (\ref{obsl}) can be taken as 
\begin{equation} \label{tlm}
L_i:=T_i\begin{bmatrix}
L_{io}\\
0
\end{bmatrix},~ M_i:=T_i\begin{bmatrix}
P_{io}^{-1} &0\\
0 &P_{iu}^{-1}
\end{bmatrix}T_i^T,
\end{equation}
where $L_{io}=P_{io}^{-1}W_i$, $i\in \mathcal{N}$.
\end{lemma}

\begin{proof}

Choose a candidate Lyapunov function for the error system (\ref{sys_eg})
\begin{equation}
V(e_1,\cdots,e_N):=\sum_{i=1}^Ne_i^TP_ie_i,
\end{equation}
where $P_i:=T_i\begin{bmatrix}
P_{io} &0\\
0 & P_{iu}
\end{bmatrix}T_i^T $. Clearly then $P_i>0$.

The time-derivative of $V$ is
\begin{equation}
\dot{V}(e)=e^T(P\Lambda+\Lambda ^TP-\gamma P\overline{M}(R\mathcal{L}\otimes I_n)-\gamma(\mathcal{L}^TR\otimes I_n)\overline{M}^TP)e
\end{equation}
where $P=\mathrm{diag}\{P_1,\cdots,P_N\}$. Since the matrix $M_i$ in (\ref{tlm}) is chosen as $M_i=P_i^{-1}$, we have $\overline{M} = P^{-1}$. Hence, the time-derivative of $V$ becomes
\begin{equation}
\dot{V}(e)=e^T(P\Lambda+\Lambda ^TP-\gamma (R\mathcal{L}+\mathcal{L}^TR)\otimes I_n)e.
\end{equation}

On the other hand, we get the following inequality by (\ref{Th1}) and (\ref{L1}) in Lemma 4.
\begin{equation}\label{Th1q}
\mathrm{diag}\{Q_1,\cdots,Q_N\}-T^T\gamma((R\mathcal{L}+\mathcal{L}^TR)\otimes I_n)T<0,
\end{equation}
where $Q_i=\begin{bmatrix}
\Phi _i &A_{ir}^TP_{iu}\\
P_{iu}A_{ir} & P_{iu}A_{iu}+A_{iu}^TP_{iu}+2\alpha P_{iu}
\end{bmatrix}$, $i\in \mathcal{N}$, with $\Phi _i$ as defined in the statement of the lemma.

Since $L_{io}=P_{io}^{-1}W_i$, pre- and post- multiplying the inequality (\ref{Th1q}) with $T$ and its transpose, we get
\begin{equation}
P\Lambda+\Lambda ^TP-\gamma(R\mathcal{L}+\mathcal{L}^TR)\otimes I_n+2\alpha P<0,
\end{equation}
which implies $\dot{V}(e)<-2 \alpha V(e)$. Hence the solutions of the error system (\ref{sys_eg}) converge to zero asymptotically with decay rate at least $\alpha$ \cite{Tanaka1998TFS}, and the distributed observer (\ref{obsl}) achieves omniscience asymptotically.

\end{proof}



Using the previous lemmas we now arrive at our main result:
\begin{theorem}
Let $\alpha>0$. If $(H,A)$ is observable and $\mathcal{G}$ is a strongly connected directed graph, then there exists a distributed observer (\ref{obsl}) that achieves omniscience asymptotically while all solutions of the error system converge to zero with decay rate at least $\alpha$. Such observer is obtained as follows:
\begin{itemize}
\item[1] For each $i\in \mathcal{N}$, choose an orthogonal matrix $T_i$ such that
\begin{equation} 
T_i^TAT_i=\begin{bmatrix}
A_{io} &0\\
A_{ir} &A_{iu}
\end{bmatrix},~ H_iT_i=\begin{bmatrix}
H_{io} &0
\end{bmatrix}
\end{equation}
with $(H_{io},A_{io})$ observable.
\item[2] Compute the positive row vector $r=\begin{bmatrix}
r_1,\cdots,r_N
\end{bmatrix} $ such that $r\mathcal{L}=0$ and $r\mathbf{1}_N=N$.
\item[3] Put $g_i=1$, $i\in \mathcal{N}$ and take $\epsilon >0$ such that (\ref{L1}) holds.
\item[4] Take $\gamma>0$ sufficiently large so that for all $i\in \mathcal{N}$ 
\begin{equation}\label{Thg1}
A_{iu}+A_{iu}^T-(\gamma \epsilon-2\alpha) I_{n-v_i}+\frac{1}{\gamma \epsilon-2\alpha}A_{ir}A_{ir}^T<0.
\end{equation}
\item[5] Choose $L_{io}$ such that all eigenvalues of $A_{io}-L_{io}H_{io}$ lie in the region $\{s\in \mathbb{C}~ |~ \mathrm{Re}(s)<-\alpha\}$. 
\item[6] For all $i\in \mathcal{N}$, solve the Lyapunov equation 
\begin{equation}\label{Th2}
\mathrm{Sym}(P_{io}(A_{io}-L_{io}H_{io}+\alpha I _{v_i}))+(\gamma-2\alpha)  I_{v_i}=0
\end{equation}
to obtain $P_{io}>0$.
\item[7] Define  
\begin{equation} \label{tlm2}
L_i:=T_i\begin{bmatrix}
L_{io}\\
0
\end{bmatrix}, M_i:=T_i\begin{bmatrix}
P_{io}^{-1} &0\\
0 &I_{n-v_i}
\end{bmatrix}T_i^T,i\in \mathcal{N}
\end{equation}
\end{itemize}
\end{theorem}

\begin{proof}
We choose $g_i=1$, $i\in \mathcal{N}$. Since the pair $(H,A)$ is observable and the graph $\mathcal{G}$ is a strongly connected directed graph, $\epsilon>0$ can be obtained by Lemma 4. 

Putting $P_{iu}=I_{n-v_i}$, $i\in \mathcal{N}$, the inequality (\ref{Th1}) in Lemma 5 becomes
\begin{equation}\label{Thg}
\begin{bmatrix}
\Phi _i+ \gamma  I_{v_i} &A_{ir}^T\\
A_{ir} &A_{iu}+A_{iu}^T+2\alpha I_{n-v_i}
\end{bmatrix}-\gamma \epsilon I_n<0,~ \forall i\in \mathcal{N}.
\end{equation}
where $\Phi _i= P_{io}A_{io}+A_{io}^TP_{io}-W_iH_{io}-H_{io}^TW_i^T+2\alpha P_{io}$.
By substituting (\ref{Th2}) and $W_i=P_{io}L_{io}$ into (\ref{Thg}), we have that the inequality (\ref{Thg}) holds if
\begin{equation}\label{Th3}
\begin{bmatrix}
-(\gamma \epsilon-2\alpha) I_{v_i} &A_{ir}^T\\
A_{ir} &\mathrm{Sym}(A_{iu})-(\gamma \epsilon -2\alpha) I_{n-v_i}
\end{bmatrix}<0,\forall i\in \mathcal{N}.
\end{equation}
By using the Schur complement lemma, (\ref{Th3}) is equivalent with
\begin{equation} \label{Th4}
A_{iu}+A_{iu}^T-(\gamma \epsilon -2\alpha) I_{n-v_i}+\frac{1}{\gamma \epsilon -2\alpha}A_{ir}A_{ir}^T<0,~ \forall i\in \mathcal{N}.
\end{equation}
As stated in step 4, inequality (\ref{Th4}) holds with sufficiently large $\gamma >0$.

Thus, we find that the parameters introduced in steps 3 to 6 guarantee that the inequality (\ref{Th1}) in Lemma 5 holds. Hence, the distributed observer (\ref{obsl}) with gain matrices $L_i$ and $M_i$ achieves omniscience asymptotically with decay rate at least $\alpha$.

\end{proof}

\begin{remark}
\rm{Since $(H_{io},A_{io})$ is observable, the Lyapunov equation (\ref{Th2}) in step 6 can indeed be made to have a solution for any $\alpha>0$. The coupling gain $\gamma >0$ can be taken sufficiently large such that (\ref{Thg1}) holds for any given $\alpha>0$, which means that the error can be made to converge to zero at any desired rate.}
\end{remark}

\begin{remark}
\rm{The design procedure in Theorem 6 gives one possible chioce of solutions of the inequality (\ref{Th1}) in Lemma 5, which also means that the inequality (\ref{Th1}) always has the required solutions under our standing assumptions that $(H,A)$ is observable and the graph $\mathcal{G}$ is strongly connected. In fact, inequalities (\ref{L1}) in Lemma 4 and (\ref{Th1}) in Lemma 5 both are LMI's, which can be solved numerically by using the LMI Toolbox or YALMIP in MATLAB directly.}
\end{remark}

\begin{remark}
\rm{In the special case that the communication graph among the observers is a connected undirected graph, we have that $r=\mathbf{1}_N^T$ is the unique positive row vector such that $r\mathcal{L}=0$ and $r\mathbf{1}_N=N$. In the design procedure of Theorem 6, we can then take $r_i=1$ for all $i \in \mathcal{N}$.}
\end{remark}

\section{Conclusions}
In this note, we have proposed a novel simple approach to distributed observer design for LTI systems. The information among the local observers is exchanged by a strongly connected directed graph. The local Luenberger observer at each node is designed to asymptotically estimate the state of the dynamical system. Each local observer has state dimension equal to that of the 
observed plant. We have analyzed the structure of the required gain matrices using an observability decomposition of the local systems. By introducing an auxiliary undirected graph, a bank of LMI's is presented to calculate the gain matrices in our distributed observer. Finally, we have recovered the known sufficient conditions for the existence of a distributed observer and have presented a simple optional algorithm to design a distributed observer.


%

\appendices
%



\ifCLASSOPTIONcaptionsoff
  \newpage
\fi



\bibliographystyle{IEEEtran}
\bibliography{IEEEabrv,mybibfiles}
\end{document}